\documentclass[11pt]{article}
\usepackage{amsmath,amsthm,amsfonts,amssymb,amscd, amsxtra}
\usepackage{graphicx,multirow}
\usepackage{float}
\usepackage[margin=2.5cm,nohead]{geometry}
\newtheorem{theorem}{Theorem}[section]

\newtheorem{lemma}{Lemma}[section]
\newtheorem{proposition}{Proposition}[section]
\newtheorem{remark}{Remark}[section]

\newtheorem{definition}{Definition}[section]

\DeclareMathOperator{\grad}{grad}

\title{Proximal Point Method for Vector Optimization on  Hadamard Manifolds}
\author{G. C. Bento\thanks{IME, Universidade Federal de Goi\'as,
Goi\^ania, GO 74001-970, BR ({\tt glaydston@ufg.br}).}
\and
 Ferreira, O. P.
\thanks{IME, Universidade Federal de Goi\'as,
Goi\^ania, GO 74001-970, BR({\tt orizon@ufg.br}).}
\and
 Pereira,  Y. R. L..
\thanks{IME, Universidade Federal de Goi\'as,
Goi\^ania, GO 74001-970, BR({\tt orizon@ufg.br}).}
}

\begin{document}

\maketitle

\begin{abstract} In this paper, we extend the proximal point algorithm for vector optimization from the Euclidean space  to the  Riemannian  context. Under suitable assumptions  on the objective function the well definition and  full  convergence  of the    method  to a weak efficient point is proved.  \\

\noindent
{\bf Keywords:}  Proximal method; Vector optimization; Fej\'er convergence;   Hadamard manifolds.

\noindent
{\bf AMS:}  90C30,  26B25, 65K05, 53C21.
\end{abstract}
\section{Introduction}

 In recent years,  extensions to Riemannian manifolds of concepts and techniques which fit in linear spaces are natural. Several  algorithms for optimization  problem which involve convexity of the  objective function have been extended from the linear settings to  the Riemannian context; see, for instance,  \cite{Bacak2013, Barani2009, FerreiraOliveira2002, HosseiniPouryayevali2013,  PapaQuirozOliveira2012, WangLi2015, WangLiYao2015}  and the references therein. One reason for the success of this extension is the possibility to transform, by introducing a suitable Riemannian metric, nonconvex problems in the linear context into convex problems in the  Riemannian context; see   \cite{BentoJefferson2012, ColaoGenaro2012,  CruzNetoFerreira2006, UdristeLivro1994}. 

In the last few  years, researchers began the study of the vector optimization  problems on Riemannian manifolds context;  papers dealing with this issues include  Bento and Cruz Neto \cite{BentoCruzNeto2013}, Bento et al. \cite{BCNS2013},  Bento et al. \cite{BFO2012} and  Bonnel et al. \cite{BonnelUd2015}. The present  paper  deals with  the extension of  the proximal point method for vector optimization from the Euclidean settings  to the  Riemannian context, which continues the subject addressed in the following papers \cite{FerreiraOliveira2002, BentoFO2015, BentoFO2010, WangLiYao2015, LiGeVic2009}.  To our best knowledge, this is the first paper extending the proximal point method for vector optimization to the Riemannian settings. Besides  our approach  is new even in  Euclidean context, since  we are dealing with general convex cone and the nonlinear  scalarization is more  flexible than the considered in \cite{BentoCNS2014}. Under suitable assumptions  on the objective function the well definition and  full  convergence  of the    method  to a weak efficient point is proved.   It is worth to point out  that under assumption of null sectional curvature, our algorithm retrieves the proximal point method for multobjective  presented in \cite{BentoCNS2014} and, somehow, goes further.

This paper is organized as follows. In Section~\ref{ssec1}, we present the notations and terminology used in the paper.   In  Section~\ref{sec2},  we present  the  vector optimization problem and  the  proximal  point method. Our main results are stated and proved in Section~\ref{sec3}, and conclusions are discussed in Section~\ref{sec4}.
\subsection{Notation and Terminology } \label{ssec1}
In this section, we introduce some notations about Riemannian geometry, which can be found in any introductory book on Riemannian geometry, such as in Sakai \cite{S1996},  Do Carmo \cite{C1992}.  Let $M$ be a $n$-dimentional Hadamard  manifold. {\it In this paper, all manifolds $M$ are assumed to be Hadamard finite dimensional}. We denote by $T_pM$ the $n$-dimentional {\it tangent space} of $M$ at $p$, by $TM=\cup_{p\in M}T_pM$ {\itshape{tangent bundle}} of $M$ and by ${\cal X}(M)$ the space of smooth vector fields over $M$.  The Riemannian  distance between $p, q\in M$,  denoted by $d(p,q)$, and given a nonempty set $U\subset M$, the distance function associated to $U$ is given by:
$$
M\ni p \mapsto d(p,U):= \inf \{d(p,q)~:~ q\in U\}.
$$
Since $M$ is a Hadamard manifold, the Riemannian distance $d$ induces the original topology on $M$, namely,  $( M, d)$ is a complete metric space and bounded and closed subsets are compact.   The open metric ball at $p\in M$  is given by 
$
B (p,r):=\{ q\in M ~: ~ d(p, q)<r\}, 
$
where $r>0$.  The Riemannian metric is denoted by  $\langle \,,\, \rangle$ and  the corresponding norm  by $\| \; \|$.  The metric induces a map $g\mapsto\grad g\in{\cal X}(M)$ which associates to each function differentiable over $M$ its gradient via the rule $\langle\grad g,X\rangle=d g(X),\ X\in{\cal X}(M)$.   The  geodesic  determined by its position $p$ and velocity $v$ at $p$ is denoted by $\gamma=\gamma _{v}(.,p)$.  The restriction of a geodesic to a  closed bounded interval is called a {\it geodesic segment}. Since $M$ is a Hadamard manifold the lenght   of the  geodesic segment  $\gamma$  joining $p$ to $q$ is equals $d(p,q)$. Moreover, {\it exponential map} $exp_{p}:T_{p}  M \to M $ is defined by $exp_{p}v\,=\, \gamma _{v}(1,p)$ is  a diffeomorphism and, consequently, $M$ is diffeomorphic to the Euclidean space $\mathbb{R}^n $, $ n=dim M $.  Let ${q}\in M $ and $exp^{-1}_{q}:M\to T_{p}M$ be the inverse of the exponential map. Note that $d({q}\, , \, p)\,=\,||exp^{-1}_{p}q||$,  the function $d_{q}^2: M\to\mathbb{R}$ defined by $ d_{q}^2(p)=d^2(q,p)$ is  $C^{\infty}$ and $ \grad d_{q}^2(p):=-2exp^{-1}_{q}{p}.$ Furthermore,  we know that
\begin{equation} \label{ineq.had.}
d^2(p_1,p_3)+d^2(p_3,p_2)-\langle \exp_{p_3}^{-1}p_1,\exp_{p_3}^{-1}p_2\rangle\leq d^2(p_1,p_2). \qquad p_1, p_2 , p_3 \in M.
\end{equation}
A set  $\Omega\subseteq M$ is said to be {\it convex}  if any geodesic segment with end points in $\Omega$ is contained in
$\Omega$, that is,  if $\gamma:[ a,b ]\to M$ is a geodesic such that $x =\gamma(a)\in \Omega$ and $y =\gamma(b)\in \Omega$; then $\gamma((1-t)a + tb)\in \Omega$ for all $t\in [0,1]$.  Let $g:M\to\mathbb{R}\cup\{+\infty\}$  be a function. The {\it domain} of $g$ is the set  $\mbox{dom}g:=\left\{ p\in M ~: ~g(p)<\infty \right\}.$ The function $g$ is said to be proper if $\mbox{dom}~g\neq \varnothing$ and  {\it convex}  (resp.  {\it strictly convex, strongly  convex}) on a convex set $\Omega\subset \mbox{dom}~g$ if for any geodesic segment $\gamma:[a, b]\to\Omega$ the composition $g\circ\gamma:[a, b]\to\mathbb{R}$ is convex  (resp.  {\it strictly convex, , strongly  convex}). It is well known that, for each $q\in M$,  $d_{q}^2$ is strongly convex. Take $p\in \mbox{dom}~g$. A vector $s \in T_pM$ is said to be a {\it subgradient\/} of $g$ at $p$, if $g(q) \geq g(p) + \langle s, \, \exp^{-1}_pq\rangle$, for $q\in M.$ The set  $\partial g(p)$ of all subgradients of $g$ at $p$ is called the {\it subdifferential\/} of $g$ at $p$. We remark that,  if $g$  is  convex   then $\partial g(p)\neq \varnothing$ for all $p\in \mbox{dom}~g$. Moreover, if  $g$ is differentiable  at $p$ then $\partial g(p)=\{\mbox{grad}f(p)\}$. The function  $g$ is {\it lower semicontinuous} (lsc) at $\overline{x}\in\mbox{dom}g$ if for each sequence $\{x^n\}$ converging to $\overline{x}$ we have $ \liminf_{n\rightarrow\infty} g(x^n)\geq g(\overline{x}).$ Given a closed set $\Omega \subset\mathbb M$, it is known that indicator function of $\Omega$, $I_\Omega:M \rightarrow \mathbb{R}\cup\{+\infty\}$, is a lower semicontinuous function and,   for each $p \in M$,  $\partial I_\Omega(p)=N_\Omega(p)$, where
$$
N_\Omega(p):=\left\{v\in T_pM ~: ~ \langle v, \; \exp^{-1}_p q\rangle \leq 0, ~q \in \Omega \right\}.
$$
\begin{definition}
A sequence $\{p^k\} \subset (M,d)$ is said to Fej\'er convergence to a set $W\subset M$ if, for every $q \in W$  we have
$d^2(w,p^{k+1})\leq d^2(w,p^k).$
\end{definition}
\begin{proposition}\label{fejer}
Let $\{p^k\}$ be a sequence in $(M,d)$. If $\{p^k\}$ is Fej\'er convergent to non-empty set $W\subset M$, then $\{p^k\}$ is bounded. If furthermore, an accumulation point $p$ of $\{p^k\}$ belongs to $W$, then $\lim_{k\rightarrow \infty}p^k=p$. 
\end{proposition}
\section{The  Proximal Point  Method for Vector Optimization  }\label{sec2}
In this section,  we present  the  vector optimization problem,  some concepts and results related to this problem,  and   introduce the  proximal  point method for this problem.

Let  $C\subset\mathbb{R}^m$  be  a closed,  pointed and  convex cone.  We will  use the  binary relations  $\preceq_C$ and $\prec_C$  defined,  respectively,  by   $p \preceq_C       q $ means  $q-p \in C$ and $p \prec_C q$ means  $q-p \in \mbox{int}C$, for all $p, q\in \mathbb{R}^m$.   Given a continuously differentiable vector function $F: M\to \mathbb{R}^m$, we consider the problem of finding an {\it  efficient point} of F, i.e., a point $p^*\in \mathbb{R}^n$ such that there exists no  $p\in M$ with $F(p) \preceq_C F(p^*)$ and $F(p) \neq F(p^*)$.  We denote this unconstrained problem as
\begin{equation}\label{VOP}
C-\mbox{Min}_{p \in M} F(p).
\end{equation}
We say that $p^*\in M$ is a {\it  weakly efficient point}  of \eqref{VOP} if there is no  $p\in M$ such that $F(p)\prec_C F(p^*)$.  The set of the   weakly efficient points of \eqref{VOP} is denoted by  $C-\mbox{argmin}_w\{F(p)|p\in M\}$.  Throughout this paper we assume that \eqref{VOP}  satisfies   the   following assumption:
\begin{description}
	\item[{\bf (A1)}] There exists a compact set $Z \subset \mathbb{R}^m\backslash \{0\}$ such that $C=\{y\in\mathbb{R}^m~: ~\langle y,z\rangle\geq 0,  ~ z\in Z\}$.
\end{description}
\begin{remark}
In classical optimization $C=\mathbb{R}_+$ and we can take $Z=\{1\}$. For multiobjective optimization, $C$ is  the positive orthant of $\mathbb{R}^m$ and we can take $Z$ as the canonical base of $\mathbb{R}^m$. For a generic cone $C$ we can take $Z=\{z \in C^*~:~\|z\|_1=1\}$,
 where   $C^*:=\{y\in\mathbb{R}^m~: ~\langle y,x\rangle\geq 0,  ~ x\in \mathbb{R}^m\}$ and   $\|z\|_1:= |z_1|+|z_2|+\cdots + |z_m|$.
 \end{remark}
We consider the following   nonlinear scalar function  $f:\mathbb{R}^m\rightarrow \mathbb{R}$, which  will play an important role in our analysis,  defined by
\begin{equation}\label{non.scal.1}
f(y):=\inf\{t\in \mathbb{R} ~: ~\,te\in y + C \},
\end{equation}
where $e$ is any fixed point in int$C$; see \cite{XL2010}. In \cite[Proposition 1.44]{CHY2005} it was proved that  the nonlinear function above can be rewritten as follow 
\begin{equation}\label{non.scal.2}
f(y)=\max_{z\in Z}\frac{\langle y,z \rangle}{\langle e,z\rangle}.
\end{equation}
\begin{remark}
For the multiobjective case,   \eqref{non.scal.1} becomes    $f(y)=\max_{i \in I}\langle y ,e_i\rangle$,  where $ \{e_i\} \subset \mathbb{R}^m$ is the canonical base of the space in $\mathbb{R}^m$, which has been used in  \cite{BentoCNS2014}.
\end{remark}
Next lemma, which proof is trivial, gives us some properties of the function above  that will be useful through  the paper.
\begin{lemma}\label{lem.scal.0}
Let $f:\mathbb{R}^m\rightarrow \mathbb{R}$ be a nonlinear  function defined in \eqref{non.scal.2}, then following properties hold:
\begin{itemize}
\item[i)]Given $y\in \mathbb{R}^m$, $\alpha\in \mathbb{R} $ and $t >0$, $f(ty+\alpha e)= tf(y)+ \alpha$;
\item[ii)]If $y\preceq_C z$ then $f(y)\leq f(z)$ for any $y,\,z \in \mathbb{R}^m$.
\end{itemize}
\end{lemma}
\begin{proposition}\label{scal.1}
Let $F: M \rightarrow \mathbb{R}^m$ be a vectorial function and $C \subset M$ a closed set. Then,
$$
\mbox{argmin}_{p\in C} f(F(p)) \subset \mbox{argmin}w_{p\in C} F(p).
$$
\end{proposition}
\begin{proof}
It is similar to the proof of  \cite[Proposition 3.1]{BentoCNS2014}.
\end{proof}
Now   we introduce the  proximal  point method for vector optimization.  Let   $\{\lambda^k\}$   be    a sequence of positive numbers and    $\{e^k\}\subset \mbox{int}\, C$ such that $\|e^k\|=1$, for $k=0,1,  \ldots$.  Consider the sequence    of functions 
\begin{equation} \label{eq:fk}
f_k(y):=\max_{z\in Z}\frac{\langle y,z \rangle}{\langle e^k,z \rangle}, \qquad  k=0,1,  \ldots.
\end{equation}
The  {\it proximal  point method}  for solving  \eqref{VOP}, with starting point  $p^0\in M$,  is defined by 
\begin{equation}\label{eq.alg1}
p^{k+1}:= \mbox{argmin}_{p \in \Omega_k} f_k\left(F(p)+\frac{\lambda_k}{2} d^2(p,p^k)e^k\right), \qquad k=0, 1, \ldots, 
\end{equation}
where $\Omega_k :=\{p\in M~:~F(p)\preceq_C  F(p^k)\}$.   From now on $\{p^k\}$ denotes the sequence generated by the  proximal  point method, with starting point  $p^0\in M$. 
\section{Convergence Analysis} \label{sec3}
In this section,  we prove  the full  convergence  of the   proximal  point method  to a weak efficient point.  For this purpose,  we  need to define the convexity of a function  with respect to the  order induced  by $C$. A vectorial  function  $F:M\rightarrow\mathbb{R}^m$ is called $C$-\emph{ convex} if, for  $p , q \in M$ and $\gamma:[0,1]\rightarrow M$ a geodesic segment joining $p$ to $q$, there holds $F(\gamma(t))\preceq_C (1-t)F(p) +	tF(q),$ for all $ t \in[0,1]$.

We also  need of the following assumption:
\begin{description}
	\item[{\bf (A2)}]  $ \bar\Omega:=\bigcap_{k=0}^{\infty} \Omega_k \neq \varnothing $.
\end{description}
\begin{remark}
In general the set $\bar\Omega$ in {\bf (A2)} can be an empty set. One way to guarantee that $\bar\Omega$ is nonempty is to assume: for   each $p^0\in M$ the set $(F(p^0)-C)\cap F(M)$ is  $C$-{\it complete}  (see \cite[Section 19]{Luc1989}),   meaning  that each sequence $\{q^k\}\subset M$, with $q^0=p^0$,  such that $F(q^{k+1})\preceq_C F(q^k)$,  for $k=0, 1,  \ldots $, there exists  $q \in M$ such that $F(q)\preceq_C F(q^k)$, for $k=0, 1,  \ldots $.  This assumption  is standard to ensure the convergence  of descent methods in vector optimization; see, for instance,  \cite{Bonnel2005, Ceng2007, Fukuda2013, DrummondSvaiter2005, Villacorta2011}.
\end{remark}
Now we ready to state and  prove  the main result of this section. 
\begin{theorem}\label{ex.eff.point}
Let $F:M\rightarrow R^m$ be a  $C$-convex  function,  and assume that   {\bf (A1)} and {\bf (A2)} hold and $\{\lambda^k\}$  is bounded.  Then,   $\{p^k\}$ is well defined and converges to a  weakly efficient point.    
\end{theorem}
\begin{proof}
Let $\{f_k\}$  be the sequence of functions defined in \eqref{eq:fk},  and define $\varphi_k : M\rightarrow \mathbb{R}\cup \{+\infty\}$ by 
$$
\varphi_k(p)=f_k\left( F(p) + I_{\Omega_k}(p)e^k + \lambda_k  d^2(p,p^{k})e^k\right),  \qquad  k=0, 1, \ldots, 
$$
where $I_{\Omega_k}$ is the indicator function of $\Omega_k$.   From  item $i$ of Lemma~\ref{lem.scal.0}   we have 
\begin{equation} \label{eq:wdee}
\varphi_k(p)= f_k\left(F(p)\right)+I_{\Omega_k}(p)+\frac{\lambda_k}{2} d^2(p,p^k), \qquad k=0, 1, \ldots.
\end{equation}
Since  $F$ is $C$- convex,  then $f_k\circ F$ is convex and    $\Omega_k$ is a convex and  closed set. Hence $\varphi_k$ is  strongly convex and   lower semicontinuous on $\Omega_k$, for $k=0, 1, \ldots $. Thus, there exists  a unique $p^{k+1} \in \Omega_k$  such that 
\begin{equation} \label{eq:wddp}
p^{k+1}=\mbox{argmin}_{p\in M}\varphi_k(p),  \qquad  k=0, 1, \ldots, 
\end{equation}
which implies that $\{p^k\}$ is well defined and the first part of the  proposition is proved. 

Using  convexity of  $\varphi_k$  and    \eqref{eq:wddp}  we conclude that   $0\in \partial \varphi_k(p^{k+1})$, which from  \eqref{eq:wdee} yields 
$$
0\in \partial \left( f_k\circ F(\cdot) + I_{\Omega_K}(\cdot) + \frac{\lambda_k}{2} d^2(\cdot,p^k)\right)(p^{k+1}),   \qquad  k=0, 1, \ldots.
$$
Last inclusion implies that  there exist  $w^{k+1}\in\partial( f_k\circ F)(p^{k+1})$ and $v^{k+1}\in N_{\Omega_k}(p^{k+1})$ such that 
\begin{equation}\label{conv.eq.2}
 \exp^{-1}_{p^{k+1}}p^k= \frac{1}{\lambda_k}\left( w^{k+1} + v^{k+1}\right),   \qquad  k=0, 1, \ldots.
\end{equation}  
On the other hand, using inequality \eqref{ineq.had.}  with $p_1=p\in M$, $p_2=p^k$ and $p_3=p^{k+1}$,  we  have
$$
d^2(p, p^{k+1})+d^2(p^{k+1},p^k)-\langle\exp^{-1}_{p^{k+1}}p,\exp^{-1}_{p^{k+1}}p^{k}\rangle\leq d^2(p, p^{k}),    \qquad  k=0, 1, \ldots.
$$
Substituting  the equality in   \eqref{conv.eq.2} into the  last  inequality,  we obtain
$$
d^2(p, p^{k+1})\leq d^2(p, p^{k})-d^2(p^{k+1},p^k)+\frac{1}{\lambda_k}\langle\exp^{-1}_{p^{k+1}}p, w^{k+1}+v^{k+1}\rangle,  \qquad  k=0, 1, \ldots.  
$$
 Since $f_k\circ F$ is convex and $w^{k+1}\in\partial( f_k\circ F)(p^{k+1})$ then  assumption {\bf (A2)} implies 
$$
\left\langle\exp^{-1}_{p^{k+1}}p,\,w^{k+1}\right\rangle\leq f_k(F(p))-f_k(F(p^{k+1}))\leq 0,  \qquad p\in  \bar\Omega,  \qquad  k=0, 1, \ldots, 
$$
where the last inequality follows from   item $ii$  of Lemma~\ref{lem.scal.0}.  Now, taking into account that $ \bar\Omega\subset \Omega_k$  and $v^{k+1}\in N_{\Omega_k}(p^{k+1})$, for $ k=0, 1, \ldots$,   we have
$$ 
\langle\exp^{-1}_{p^{k+1}}p, \, v^{k+1}\rangle\leq0,    \qquad p\in  \bar\Omega,  \qquad  k=0, 1, \ldots.
$$
Therefore,   taking into account  {\bf (A2)}, we can  combine three last inequalities to conclude  that  $\{p^{k}\}$  is Fej\'er convergence to $ \bar\Omega$.  In particular,     Proposition~\ref{fejer}  implies that   $\{p^{k}\}$ is bounded. Let $\bar{p}$  be a  cluster point of  $\{p^k\}$ and  $\{p^{k_j}\}$ a subsequence   of     $\{p^k\}$  such that $\lim_{k\rightarrow\infty} p^{k_j} = \bar{p}$. Note that \eqref{eq.alg1} yields 
$$
F(p^{k+1})\preceq_C F(p^k), \qquad   k=0, 1, \ldots.
$$
Hence, the continuity of  $F$   implies   that   $F(\bar{p})\preceq_C F(p^k)$, for  $k=0, 1, \ldots$, which is equivalent to say  $\bar{p}\in \bar\Omega$.   Using  Proposition~\ref{fejer} we conclude that  $\{p^k\}$ converges to $\bar{p}$.

It remains to prove that $\bar{p}$ is a weakly efficient point.  Let us suppose,  by contradiction,  that  there exists $\hat{p}\in M$ such that $F(\hat{p})\prec_C F(\bar{p})$ (note that,  in particular,    $\hat{p}\in \Omega$).   Thus,  since   $\bar{p}\in \bar\Omega$, using item $ii)$ of Lemma~\ref{lem.scal.0}   we have 
\begin{equation}\label{conv.eq.4}
f_k(F(\hat{p}))-f_k(F(\bar{p}))\geq f_k(F(\hat{p}))-f_k(F(p^{k+1}))\geq\langle w^{k+1}, \exp^{-1}_{p^{k+1}}\hat{p}\rangle,
\end{equation} 
where the last inequality was obtained by  using  the convexity of  the function   $f_k\circ F$  and  that  $w^{k+1} \in \partial (f_k\circ F)(p^{k+1})$. Using \eqref{conv.eq.2} we obtain 
$$
\langle w^{k+1},\exp^{-1}_{p^{k+1}}\hat{p}\rangle=-\langle v^{k},\exp^{-1}_{p^{k+1}}\hat{p}\rangle +\lambda_{k}\langle\exp^{-1}_{p^{k+1}}p^{k},\exp^{-1}_{p^{k+1}}\hat{p}\rangle. 
$$
Using the definition of $f_k$ we have $f_k(F(\bar{p})-F(\hat{p}))\geq f_k(F(\bar{p}))-f_k(F(\hat{p}))$, for $k=0,1, \ldots$. Thus, combining the last equality  with \eqref{conv.eq.4} and taking into account that $v^{k+1}\in N_{\Omega_k}(p^{k+1})$   we have 
$$
f_k(F(\bar{p})-F(\hat{p}))\geq f_k(F(\bar{p}))-f_k(F(\hat{p}))\geq \lambda_{k}\langle\exp^{-1}_{p^{k+1}}p^{k},\exp^{-1}_{p^{k+1}}\bar{p}\rangle.
$$
Hence, using again the  definition of $f_k$, the last inequality becomes 
$$
\max_{z\in Z} \frac{\langle F(\bar{p})-F(\hat{p}),z \rangle}{\langle e_k,z\rangle}\geq \lambda_{k}\langle\exp^{-1}_{p^{k+1}}p^{k},\exp^{-1}_{p^{k+1}}\bar{p}\rangle.
$$
Since $Z$ is a compact set, then  there exists $\bar{z}\in Z$  such that 
$$ 
 \langle F(\bar{p})-F(\hat{p}),\bar{z} \rangle\geq \lambda_{k}\langle\exp^{-1}_{p^{k+1}}p^{k},\exp^{-1}_{p^{k+1}}\bar{p}\rangle \langle e_k,\bar{z}\rangle.
$$ 
Note that the sequences  $\{ \langle e_k,\bar{z}\rangle\}$ and  $\{\lambda^k\}$ are bounded. Thus, letting $k$ goes to infinity in the  last inequality, we have
$$
 \langle F(\bar{p})-F(\hat{p}),\bar{z}\rangle \geq 0, 
$$ 
which  contradicts the fact that $F(\bar{p})\prec_C F(\hat{p})$ and the desired result follows.  
\end{proof}
\section{ Final Remarks} \label{sec4}
It is worth to point out that the nonlinear scalar function, see \eqref{eq:fk},  considered in the iterative step process of the algorithm, see \eqref{eq.alg1},  allows a relationship between the weak sharp minima set of the vectorial optimization problem  and the weak sharp minima set of the scalarized problem. For state this  relationship,   we need some definitions and results.     Let $G:M\rightarrow R^m$,  $\eta \in \mathbb{R}^m$ and let us define  the following level set 
$$
W_{\eta}:=\left\{p\in M ~:~ G(p)=\eta \right\}.
$$
We denote by $\mbox{Min}G$  (resp. $\mbox{WMin} G$) the {\it set of the efficient points} (resp. {\it weak efficient points}) associated to \eqref{VOP}. 
\begin{definition} \label{def:wsm}
A point   $\hat{p}\in M$  is said to be  weak sharp minimum to \eqref{VOP},  if there is a constant $\tau>0$ such that 
\begin{equation} \label{eq:wsm1}
G(p)-G(\hat{p})\notin B(0,\tau d(p,W_{G(\hat{p})}))-C, \qquad p\in M\backslash W_{G(\hat{p})}, 
\end{equation}
The set  of all    weak sharp minimum to \eqref{VOP} is denoted by $ \emph{WSMin}_G $.
\end{definition}
The above  definition has appeared in several  contexts, see for example, \cite{Bednarczuk2007, BentoCNS2014, Studniarski2007, XL2010}.  Note that the relationship  \eqref{eq:wsm1} can be expressed in following equivalent form
$$
d(G(p)-G(\hat{p}),-C)\geq \tau d(p,W_{G(\hat{p})}), \qquad p\in M, 
$$
and  there holds $ \emph{WSMin}_G \subset \mbox{Min}G $. In the particular case $m$=1 and $C=\mathbb{R}_{+}$, the last inequality  becomes to the well-known  inequality 
$$
G(p)-G(\hat{p})\geq \tau d(p,W_{G(\hat{p})}),  \qquad p\in M, 
$$ 
introduced in  \cite{LiMordukhovich2011},  defining   weak sharp minimizer in  Riemannian context.

Next result establishes the above mentioned relationship  between $\mbox{WSMin}F$  and the   weak sharp minimum associated  to the  nonlinear scalar function defined  in \eqref{non.scal.1}, the proof  follows by  using  similar arguments used in the proof of \cite[Theorem 3.4]{XL2010}.
\begin{theorem}\label{sharp.1}
Let $F:M\rightarrow R^m$ and $\hat{p} \in M$. Suppose that $W_{F(\hat{p})}$ is closed set and define $\tilde{F}: M\rightarrow \mathbb{R}^n$ by $\tilde{F}(p)=F(p)-F(\hat{p})$.   If  $  \hat{p}\in   \emph{WSMin}_F$ then  $\hat{p}\in  \emph{WSMin}_ {f\circ \tilde{F}}$, where $f$ is given by  \eqref{non.scal.2}. 
\end{theorem}
We expect that the Theorem~\ref{sharp.1} constitutes a first step towards  to establish the following result:  ``If $\hat{p} \in  \emph{WSMin}_F$, then   $\{p^k\}$ converges, in a finite number of iterations". We foresee further progress along these line in the nearby future. Similar result has been proven in the  Euclidean context; see \cite{BentoCNS2014}.

\begin{thebibliography}{10}
\bibitem{Bacak2013}
Ba{\v{c}}{\'a}k, M.: The proximal point algorithm in metric spaces.
\newblock Israel J. Math. \textbf{194}(2), 689--701 (2013).

\bibitem{Barani2009}
Barani, A., Pouryayevali, M.R.: Invariant monotone vector fields on
  {R}iemannian manifolds.
\newblock Nonlinear Anal. \textbf{70}(5), 1850--1861 (2009).

\bibitem{Bednarczuk2007}
Bednarczuk, E.: On weak sharp minima in vector optimization with applications
  to parametric problems.
\newblock Control Cybernet. \textbf{36}(3), 563--570 (2007)

\bibitem{BentoCruzNeto2013}
Bento, G.C., Cruz~Neto, J.X.: A subgradient method for multiobjective
  optimization on {R}iemannian manifolds.
\newblock J. Optim. Theory Appl. \textbf{159}(1), 125--137 (2013).

\bibitem{BCNS2013}
Bento, G.C., da~Cruz~Neto, J.X., Santos, P.S.M.: An inexact steepest descent
  method for multicriteria optimization on {R}iemannian manifolds.
\newblock J. Optim. Theory Appl. \textbf{159}(1), 108--124 (2013).


\bibitem{BentoCNS2014}
Bento, G.C., Cruz~Neto, J.X., Soubeyran, A.: A proximal point-type method for
  multicriteria optimization.
\newblock Set-Valued Var. Anal. \textbf{22}(3), 557--573 (2014).

\bibitem{BentoFO2010}
Bento, G.C., Ferreira, O.P., Oliveira, P.R.: Local convergence of the proximal
  point method for a special class of nonconvex functions on {H}adamard
  manifolds.
\newblock Nonlinear Anal. \textbf{73}(2), 564--572 (2010).

\bibitem{BFO2012}
Bento, G.C., Ferreira, O.P., Oliveira, P.R.: Unconstrained steepest descent
  method for multicriteria optimization on {R}iemannian manifolds.
\newblock J. Optim. Theory Appl. \textbf{154}(1), 88--107 (2012).


\bibitem{BentoFO2015}
Bento, G.C., Ferreira, O.P., Oliveira, P.R.: Proximal point method for a
  special class of nonconvex functions on {H}adamard manifolds.
\newblock Optimization \textbf{64}(2), 289--319 (2015).


\bibitem{BentoJefferson2012}
Bento, G.C., Melo, J.G.: Subgradient method for convex feasibility on
  {R}iemannian manifolds.
\newblock J. Optim. Theory Appl. \textbf{152}(3), 773--785 (2012).


\bibitem{Bonnel2005}
Bonnel, H., Iusem, A.N., Svaiter, B.F.: Proximal methods in vector
  optimization.
\newblock SIAM J. Optim. \textbf{15}(4), 953--970 (electronic) (2005).


\bibitem{BonnelUd2015}
Bonnel, H., Todjihounde, L., Udriste, C.: Semivectorial bilevel optimization on
  riemannian manifolds.
\newblock Journal of Optimization Theory and Applications. \textbf{0}(0), 1--23
  (2015).


\bibitem{C1992}
do~Carmo, M.P.: Riemannian geometry.
\newblock Mathematics: Theory \& Applications. Birkh\"auser Boston, Inc.,
  Boston, MA (1992).


\bibitem{Ceng2007}
Ceng, L.C., Yao, J.C.: Approximate proximal methods in vector optimization.
\newblock European J. Oper. Res. \textbf{183}(1), 1--19 (2007).


\bibitem{CHY2005}
Chen, G.y., Huang, X., Yang, X.: Vector optimization, \emph{Lecture Notes in
  Economics and Mathematical Systems}, vol. 541.
\newblock Springer-Verlag, Berlin (2005).
\newblock Set-valued and variational analysis

\bibitem{ColaoGenaro2012}
Colao, V., L{\'o}pez, G., Marino, G., Mart{\'{\i}}n-M{\'a}rquez, V.:
  Equilibrium problems in {H}adamard manifolds.
\newblock J. Math. Anal. Appl. \textbf{388}(1), 61--77 (2012).

\bibitem{CruzNetoFerreira2006}
Da~Cruz~Neto, J.X., Ferreira, O.P., P{\'e}rez, L.R.L., N{\'e}meth, S.Z.:
  Convex- and monotone-transformable mathematical programming problems and a
  proximal-like point method.
\newblock J. Global Optim. \textbf{35}(1), 53--69 (2006).

\bibitem{FerreiraOliveira2002}
Ferreira, O.P., Oliveira, P.R.: Proximal point algorithm on {R}iemannian
  manifolds.
\newblock Optimization \textbf{51}(2), 257--270 (2002).


\bibitem{Fukuda2013}
Fukuda, E.H., Gra{\~n}a~Drummond, L.M.: Inexact projected gradient method for
  vector optimization.
\newblock Comput. Optim. Appl. \textbf{54}(3), 473--493 (2013).


\bibitem{DrummondSvaiter2005}
Gra{\~n}a~Drummond, L.M., Svaiter, B.F.: A steepest descent method for vector
  optimization.
\newblock J. Comput. Appl. Math. \textbf{175}(2), 395--414 (2005).


\bibitem{HosseiniPouryayevali2013}
Hosseini, S., Pouryayevali, M.R.: Nonsmooth optimization techniques on
  {R}iemannian manifolds.
\newblock J. Optim. Theory Appl. \textbf{158}(2), 328--342 (2013).


\bibitem{LiGeVic2009}
Li, C., L{\'o}pez, G., Mart{\'{\i}}n-M{\'a}rquez, V.: Monotone vector fields
  and the proximal point algorithm on {H}adamard manifolds.
\newblock J. Lond. Math. Soc. (2) \textbf{79}(3), 663--683 (2009).


\bibitem{LiMordukhovich2011}
Li, C., Mordukhovich, B.S., Wang, J., Yao, J.C.: Weak sharp minima on
  {R}iemannian manifolds.
\newblock SIAM J. Optim. \textbf{21}(4), 1523--1560 (2011).


\bibitem{Luc1989}
Luc, D.T.: Theory of vector optimization, \emph{Lecture Notes in Economics and
  Mathematical Systems}, vol. 319.
\newblock Springer-Verlag, Berlin (1989)

\bibitem{PapaQuirozOliveira2012}
Papa~Quiroz, E.A., Oliveira, P.R.: Full convergence of the proximal point
  method for quasiconvex functions on {H}adamard manifolds.
\newblock ESAIM Control Optim. Calc. Var. \textbf{18}(2), 483--500 (2012).

\bibitem{S1996}
Sakai, T.: Riemannian geometry, \emph{Translations of Mathematical Monographs},
  vol. 149.
\newblock American Mathematical Society, Providence, RI (1996).
\newblock Translated from the 1992 Japanese original by the author

\bibitem{Studniarski2007}
Studniarski, M.: Weak sharp minima in multiobjective optimization.
\newblock Control Cybernet. \textbf{36}(4), 925--937 (2007)

\bibitem{UdristeLivro1994}
Udri{\c{s}}te, C.: Convex functions and optimization methods on {R}iemannian
  manifolds, \emph{Mathematics and its Applications}, vol. 297.
\newblock Kluwer Academic Publishers Group, Dordrecht (1994).


\bibitem{Villacorta2011}
Villacorta, K.D.V., Oliveira, P.R.: An interior proximal method in vector
  optimization.
\newblock European J. Oper. Res. \textbf{214}(3), 485--492 (2011).


\bibitem{WangLi2015}
Wang, J., Li, C., Lopez, G., Yao, J.C.: Convergence analysis of inexact
  proximal point algorithms on {H}adamard manifolds.
\newblock J. Global Optim. \textbf{61}(3), 553--573 (2015).


\bibitem{WangLiYao2015}
Wang, X.M., Li, C., Yao, J.C.: Subgradient projection algorithms for convex
  feasibility on {R}iemannian manifolds with lower bounded curvatures.
\newblock J. Optim. Theory Appl. \textbf{164}(1), 202--217 (2015).


\bibitem{XL2010}
Xu, S., Li, S.J.: Weak {$\psi$}-sharp minima in vector optimization problems.
\newblock Fixed Point Theory Appl. pp. Art. ID 154,598, 10 (2010)

\end{thebibliography}

\end{document}